\documentclass[11pt,reqno,twoside]{amsart}

\usepackage[english]{babel} 
\usepackage{esint}
\usepackage[T1]{fontenc} 
\usepackage{amsmath}
\usepackage{amssymb} 
 \usepackage{enumitem} 
\usepackage{float}
\usepackage{amsfonts}
\usepackage{mathtools}
\usepackage{xcolor}
\usepackage[utf8]{inputenc}
\usepackage[mathcal]{eucal} 
\usepackage{amsthm} 
\usepackage{hyperref}
\usepackage[totalwidth=14cm,totalheight=20cm, hmarginratio=1:1]{geometry}
\usepackage[parfill]{parskip}

\DeclareMathOperator{\Hom}{Hom}

\newcommand{\R}{\mathbb{R}}

\newcommand{\SO}{\mathrm{SO}}
\newcommand{\h}{\mathbb{H}}
\newcommand{\PSL}{\mathrm{PSL}}
\newcommand{\dPSL}{\mathrm{PSL}(2,\R)\times \mathrm{PSL}(2,\R)}
\newcommand{\Sp}{\mathrm{Sp}}

\newcommand{\Pp}{\mathbb{P}}
\newcommand{\Ind}{\mathrm{Ind}}
\newcommand{\SL}{\mathrm{SL}}

\newcommand{\Max}{\mathrm{Max}}
\newcommand{\ad}{\mathrm{ad}}
\newcommand{\Hit}{\mathrm{Hit}}
\newcommand{\Core}{\mathrm{Core}}

\renewcommand{\Im}{\operatorname{Im}}
\renewcommand{\Re}{\operatorname{Re}}

\theoremstyle{definition}
\newtheorem{deff}{Definition}[section]

\newtheorem{thm}[deff]{Theorem}
\newtheorem{conj}[deff]{Conjecture}
\newtheorem{rmk}[deff]{Remark}
\newtheorem{cor}[deff]{Corollary} 
\newtheorem{lemma}[deff]{Lemma}
\newtheorem{claim}[deff]{Claim}
\newtheorem{prop}[deff]{Proposition}

\newtheorem{bigthm}{Theorem}

\DeclareMathAlphabet{\mathpzc}{OT1}{pzc}{m}{it}

\begin{document}

\setlength{\parskip}{.0em}

\title[A closed ball compactification via cores of trees]
% Cores of trees, mixed structures\\ and Thurston's compactification
{A closed ball compactification\\of a maximal component via cores of trees}
\author{Giuseppe Martone, Charles Ouyang, and Andrea Tamburelli}

\begin{abstract} We show that, in the character variety of surface group representations into the Lie group $\PSL(2,\R) \times \PSL(2,\R)$, the compactification of the maximal component introduced by the second author is a closed ball upon which the mapping class group acts. We study the dynamics of this action. Finally, we describe the boundary points geometrically as $(\overline{A_{1} \times A_{1}},2)$-valued mixed structures.
\end{abstract}

\maketitle
\setcounter{tocdepth}{1}
\tableofcontents

\setlength{\parskip}{.5em}
\section{Introduction}

A recurring theme in higher Teichm{\"u}ller theory is to relate surface group representations into higher rank Lie groups with geometric objects. Taking its cue from classical Teichm{\"u}ller theory, one is often interested in studying the degeneration of these associated geometric objects when the representation leaves all compact sets in the character variety. The celebrated Thurston compactification of Teichm{\"u}ller space regards Fuchsian representations as marked hyperbolic metrics,  where degenerating families of hyperbolic metrics subconverge to projectivized measured laminations.
One key aspect of this compactification is that it is a closed ball upon which the mapping class group acts. In years following, there have been numerous different perspectives of the Thurston compactification, using a variety of methods, topological, geometric, analytic and algebraic (see \cite{Bonahon_currents, bestvina1988degenerations, paulin1988degenerations, Wolf_harmonic, morgan1984valuations, Brumfiel_realspec}).

When the Lie group $\PSL(2,\mathbb{R})$ is replaced with a higher rank one, the relevant geometric object is not always immediately clear. In rank 2 however, combined work of Schoen \cite{Schoenharmonic}, Labourie \cite{Labourie_cyclic}, Loftin \cite{Loftin_thesis}, Collier \cite{collier2016maximal}, Alessandrini-Collier \cite{AC_SP4}, and Collier-Tholozan-Toulisse \cite{collier2019geometry}, provides a geometric interpretation to representations in the various distinguished components of the relevant character variety. These components are usually maximal components or Hitchin components, which maximize a topological quantity, the Toledo invariant, or contain a deformation of the classical Teichm{\"u}ller space. Parreau \cite{parreau_compact} compactifies them by attaching at infinity surface group actions on a Euclidean building.

This paper will primarily be concerned with the rank 2 semi-simple split Lie group $G=\PSL(2,\mathbb{R}) \times \PSL(2,\mathbb{R})$. The product structure of $G$ makes our study more amenable towards techniques from classical Teichm{\"u}ller theory. For $S$ a closed, orientable, smooth surface of genus $g>1$, work of Goldman \cite{Goldman_topcomp} shows the connected components of the character variety $\chi(\pi_{1}(S), \PSL(2, \mathbb{R}))$ are determined by the Euler number. In particular, the distinguished component with maximal Euler number of $2g-2$ is the Teichm{\"u}ller space Teich$(S)$. If we denote the character variety for $G= \dPSL$ by $\chi(\pi_{1}(S), G)$, then the connected components are merely products of the connected components of $\chi(\pi_{1}(S), \PSL(2, \mathbb{R}))$. The maximal component Max$(S,G)$ of $\chi(\pi_{1}(S), G)$ is the collection of conjugacy classes of pairs of representations, each of which is a Fuchsian representation. Hence $\mathrm{Max}(S):=\mathrm{Max}(S, \dPSL)$ is the product of two copies of Teichm{\"u}ller space.

Elements in the component Max$(S)$ have a number of related geometric interpretations. Schoen \cite{Schoenharmonic} has shown these representations correspond to equivariant minimal Lagrangians in $\mathbb{H}^{2} \times \mathbb{H}^{2}$. At the same time, the group $G=\dPSL$ is the isometry group of AdS$^{3}$, and Mess \cite{Mess} has shown the holonomy representations of GHMC-AdS$^{3}$ manifolds are precisely the ones in Max$(S)$. Krasnov-Schlenker \cite{KS_harmonicmaps} have shown to each GHMC-AdS$^{3}$ manifold, there is a unique equivariant space-like maximal surface, whose image under the Gauss map is the aforementioned minimal Lagrangian.

In seeking a compactification of Max$(S)$ via degeneration of geometric objects, the second author in his thesis \cite{Charles_dPSL} showed the natural limits to the minimal Lagrangians were given by cores of $\mathbb{R}$-trees dual to measured laminations. These are topologically and group-theoretically defined distinguished subcomplexes of the product of two trees, where some parts are two-dimensional and the remaining parts are one-dimensional. Denote by Core$(\mathcal{T}, \mathcal{T})$, the space of cores in the product of trees dual to measured laminations. Observe that there is a natural $\mathbb R^+$-action on Core$(\mathcal{T}, \mathcal{T})$ and denote by $\mathbb P\mathrm{Core}(\mathcal{T}, \mathcal{T})$ the resulting projectivization. We equip Max$(S)$ and $\mathbb P\mathrm{Core}(\mathcal{T}, \mathcal{T})$ with the equivariant Gromov-Hausdorff topology. One natural question one might ask is what exactly is the topology of the resulting compactification. Our first main result is the following.

\begin{bigthm}\label{thm:closedballIntro}
The disjoint union
\[
\mathfrak B=\mathrm{Max}(S)\sqcup \mathbb{P}\mathrm{Core}(\mathcal T,\mathcal T)
\]
is homeomorphic to a closed ball of dimension $12g-12$.
\end{bigthm}

It is not too difficult to see from the construction of this compactification, that the action of the mapping class group extends continuously to the boundary. Following Thurston, we study the action of the mapping class group  $\mathrm{MCG}(S)$ on our compactification $\mathfrak{B}$.

\begin{prop}\label{prop:upstairsfixedpts} Suppose $\phi\in\mathrm{MCG}(S)$ and $\phi(x)=x$ for some $x=(x_1,x_2)\in\mathfrak B$.
\begin{enumerate}
    \item If $\phi$ is periodic, then $x_1$ and $x_2$ are any two points fixed by $\phi$ in the Thurston compactification of Teichm\"uller space such that $(x_1,x_2)\in\mathfrak B$.
    \item If $\phi$ is pseudo-Anosov, then $(x_1,x_2)\in\partial \mathfrak B$ and $x_1=0$, or $x_2=0$ or $x_1=x_2$.
\end{enumerate}
\end{prop}

% \begin{prop} Fix $\phi\in\mathrm{MCG}(S)$. If $x\in\mathfrak B$ is such that $\phi(x)=x$, then
% \begin{enumerate}
%   {\color{magenta} \item if $x=(x_1,x_2)\in\mathrm{Max}(S)$ and $\phi(x_i)=x_i$ for $i=1,2$,
% % 	\item or $x$ is projectively equal to a pair of measured laminations $(x_1,x_2)$ such that $x_1+x_2$ is also a measured lamination. In particular, 
%     \item If $x$ is projectively equal to a pair of measured laminations $(x_1,x_2)$ and $\phi$ is pseudo-Anosov, then $x_1=0$, $x_2=0$ or $x_1=x_2$.}
% \end{enumerate}
% \end{prop}

The action of the mapping class group appears to be more interesting if we consider its action on a natural quotient of $\mathfrak{B}$. In fact, given a maximal representation $\rho$, there is a unique equivariant minimal Lagrangian $\widetilde{\Sigma}_{\rho}$ in $\h^{2}\times \h^{2}$. The induced metric on $\widetilde{\Sigma}_{\rho}$ descends to a negatively curved Riemannian metric on $S$. We denote by $\Ind(S)$ the space of such metrics. It turns out that $\Ind(S)=\Max(S)/S^{1}$. Similarly, the distance on the core of the product of two trees dual to a pair of measured laminations can be recovered from a mixed structure, that is a hybrid geometric object on $S$ that is in part a measured laminations and in part a finite area flat metric induced by a meromorphic quadratic differential on subsurfaces glued along annuli. The space of projectivized mixed structure can then be identified with the boundary of $\Ind(S)$ in the length spectrum topology (\cite{Charles_dPSL}). The mapping class group acts on $\overline{\Ind(S)}$ and we prove the following:

\begin{bigthm}\label{thm:actionMix} Assume $\phi\in\mathrm{MCG}(S)$ fixes $\mu\in\partial\overline{\mathrm{Ind}(S)}$. 
\begin{enumerate}
% 	\item If $\mu$ is {\em purely laminar}, then $\phi$ is not periodic.
	\item If $\mu$ is {\em purely flat}, then $\phi$ is periodic.
	\item If $\mu$ is {\em properly mixed}, then $\phi$ is not pseudo-Anosov.
\end{enumerate}
\end{bigthm}
Note that the remaining case of $\mu$ a {\em purely laminar} mixed structure is handled by Nielsen-Thurston's classification theorem. Theorem \ref{thm: reducible properly mixed} will give a more detailed description of item (2) in Theorem \ref{thm:actionMix} when $\phi$ is reducible. In particular, we will show that the subdivision of $S$ induced by $\mu$ is a refinement of the one induced by $\phi$ if $\mu$ has no trivial parts.

The absence of a product structure for the other simple split Lie groups of rank 2 makes the study of the topology of any compactification considerably more difficult. Furthermore, for $\dPSL$, quadratic differentials are intimately related to pairs of measured laminations, and for higher order differentials, which appear for the other rank 2 cases, there are no obvious analogous topological objects. However, it is possible to describe our compactification without explicit references to $\mathbb{R}$-trees, and we conjecture this perspective can be extended to the other rank 2 Lie groups. In particular, given any Lie algebra $\mathfrak{g}$ with Cartan subalgebra $\mathfrak{a}$ and positive Weyl chamber $\mathfrak{a}^{+}$, we define $\overline{\mathfrak{a}^{+}}$-valued measured laminations and $(\overline{\mathfrak{a}^{+}}, k)$-mixed structures obtained by gluing these vector-valued laminations together with $1/k$-translation surfaces of finite area along annuli. We can rephrase our main result as follows:

\begin{bigthm}\label{thm:Weylmixed} The boundary of $\Max(S)$ can be identified with the space of $(\overline{A_{1}^{+}\times A_{1}^{+}},2)$-mixed structures on $S$, which is thus topologically a sphere of dimension $12g-13$. 
\end{bigthm}

Moreover, we prove in Lemma \ref{lm:core_as_mixed} that $(\overline{A_{1}^{+}\times A_{1}^{+}},2)$-mixed structures are dual to the subcomplexes of an Euclidean building introduced and studied in \cite{parreau_invariant}.
Theorem \ref{thm:Weylmixed} has the advantage to be easily adaptable to other higher Teichm\"uller components (see Conjecture \ref{conj} for the precise statements in rank $2$).

\subsection*{Historical remarks} In analogy with the classical case, compactifications of higher Teichm\"uller spaces are fruitfully studied using different techniques and perspectives. As mentioned above, Parreau \cite{parreau_compact} compactifies the character variety of surface group representations into noncompact semisimple connected
real Lie groups with finite center using Euclidean buildings. For Hitchin and maximal connected components, one can obtain additional information on the boundary points by using the ($\Theta$-)positivity properties of the representations as in \cite{alessandrini_agt, burger-pozzetti,FG_main,Le2016, martone2018,martone2019sequences, parreau_invariant}. For rank two Lie groups, the second and third authors used analytic methods to study degenerations of geometric objects associated to these representations in \cite{Charles_dPSL,OT, OT_Sp4}. In a series of papers Burger, Iozzi, Parreau and Pozzetti \cite{BIPP,BIPPmetric,BIPP_positive} use geodesic currents and real algebraic geometric methods to study compactifications of general character varieties. We refer to \cite{BIPP_compte} for an account of their work and point of view. Furthermore, in independent work \cite{BIPP_personal} they describe a compactification of the product of $n$ copies of Teichm\"uller space via the projectivization of the product of $n$ copies of the space of measured laminations and vector valued mixed structures. This will lead to a (a priori different) compactification using projectivized cores when $n=2$.

% study the compactification of the space of maximal representations into the Lie group $\PSL(2,\mathbb R)^n$ for $n\geq 2$ and, for $n=2$, they obtain a statement similar to our Theorem \ref{thm:Weylmixed}.}

\subsection*{Acknowledgements} Part of this work was carried out when the first and second authors were visiting Rice University during Summer 2021. We thank the department for their hospitality. We thank Francis Bonahon for helpful comments on an earlier version of this manuscript. We thank Beatrice Pozzetti for helpful comments on this manuscript and for pointing out a mistake in a previous version of the
statement of Theorem B and Lemma 6.7. The third author acknowledges support from the National Science Foundation under grant NSF-DMS:2005501.

\section{Background}

\subsection{Foliations, laminations and $\mathbb{R}$-trees}
We recall some classical facts about measured foliations and laminations. This material can be found in \cite{fathi2021thurston}. Let $S$ be a closed, orientable, smooth surface of genus $g>1$. A \textit{measured foliation} is a singular foliation (with $k$-pronged singularities) equipped with a measure on transverse arcs, invariant under transverse homotopy. 

If $S$ is given a hyperbolic metric $\sigma$, then a \textit{measured lamination} is a closed set of disjoint simple geodesics on $(S, \sigma)$ together with a transverse measure. There is a natural homeomorphism between the space $\mathcal{MF}(S)$ of measured foliations on $S$ and the space $\mathcal{ML}(S)$ of measured laminations on $(S, \sigma)$, so that the role of $\sigma$ is an auxiliary one. Thurston showed $\mathcal{MF}(S)$ is topologically trivial, being a ball of dimension $6g-6$. The space $\mathbb{P}\mathcal{MF}(S)$ is the boundary of Teichm{\"u}ller space under the Thurston compactification.

If $S$ is given a complex structure $J$, then to any holomorphic quadratic differential $q=q(z)dz^{2}$, one may consider the foliation obtained by integrating the line field $q(v,v) >0.$ When further given the transverse measure defined by $\int_{\alpha}| \Im (\sqrt{q})|$, the resulting measured foliation is called the \textit{horizontal foliation} of $q$. Likewise integrating the line field $q(v,v) <0$ and taking the measure $\int_{\alpha}| \Re (\sqrt{q})|$ gives the \textit{vertical foliation} of $q$. The theorem of Hubbard-Masur \cite{hubbard1979quadratic} states that for a fixed Riemann surface $(S,J)$ and any measured foliation $\mathcal{F}$ on $S$, there is a unique holomorphic quadratic differential $q$, whose horizontal foliation is Whitehead equivalent (i.e. it differs at most by isotopies or expanding or collapsing pronged singularities along straight arcs)  to $\mathcal{F}$. %\AT{Whitehead equivalence has not been defined}

Any measured foliation $\mathcal{F}$ on $S$ lifts to a measured foliation $\widetilde{\mathcal{F}}$ on the universal cover $\widetilde{S}$. Taking the leaf space of $\widetilde{\mathcal{F}}$  together with a distance induced by the pushforward of the transverse measure gives an $\mathbb{R}$-tree. When an $\mathbb{R}$-tree is constructed from a measured foliation in this way, the $\mathbb{R}$-tree comes equipped with a $\pi_1(S)$-action from $\widetilde{\mathcal{F}}$. This action is \textit{small}, that is, the stabilizer of an arc never contains a free group of rank 2, and \textit{minimal}, that is, the action does not fix any proper subtree. A result of Skora \cite{skora1996splittings} says that any $\mathbb{R}$-tree with a $\pi_1(S)$-action which is both small and minimal is constructed from a measured foliation on $S$. Such $\mathbb{R}$-trees are said to be dual to a measured foliation, and for our purposes, all $\mathbb{R}$-trees we consider will be dual to a measured foliation.

\subsection{Half-translation surfaces, flat metrics and mixed structures}

A Riemann surface equipped with a holomorphic quadratic differential $q$ is called half-translation surface. This terminology comes from the fact these can be realized by gluing polygons in $\mathbb{C}$ via translations or rotations of angle $\pi$.

A half-translation surface is naturally endowed with a singular flat metric $|q|$, where the singularities are at the zeros of $q$.  Duchin-Leininger-Rafi \cite{DLR_flat} have studied the degeneration of unit-area quadratic differential metrics, and have shown the limits are precisely projectivized (quadratic) mixed structures. A \textit{mixed structure} is a collection of integrable meromorphic quadratic differential metrics on subsurfaces and measured laminations on other subsurfaces, glued along flat annuli to recover the surface $S$. Trivial examples of mixed structures include singular flat metrics on $S$ and measured laminations on $S$. We say that a mixed structure is \emph{properly mixed} if it has a flat piece but it is not a singular flat metric. Mixed structures, when the meromorphic differential is cubic or quartic, appear in the compactification of Hitchin components for $\SL(3, \mathbb{R})$ and $\Sp(4, \mathbb{R})$ (see \cite{OT}, \cite{OT_Sp4}).

A measured lamination $\lambda$ on $S$ is said to \textit{fill} if the complement $S \setminus \lambda$ is a disjoint union of topological disks. A pair $\mathcal{F}_{1}, \mathcal{F}_{2}$ of measured foliations on $S$ is said to \textit{fill} or is \textit{transverse} if for any third foliation $\mathcal{G}$, one has $i(\mathcal{F}_{1}, \mathcal{G}) + i(\mathcal{F}_{2}, \mathcal{G}) >0$. Here $i(\cdot, \cdot)$ denotes the Bonahon intersection pairing, which generalizes the topological intersection number between curves. We remark that the intersection number for the corresponding measured laminations is the same, therefore we can define filling for a pair of measured laminations analogously. Notice that given a holomorphic quadratic differential $q$, the vertical and horizontal foliations of $q$ fill. Conversely, the result by Gardiner-Masur \cite{GM_extremal_length} says that given any pair of filling measured foliations, there exists a unique Riemann surface structure and a unique holomorphic quadratic differential which realizes the original pair as its vertical and horizontal foliation (up to Whitehead equivalence). In particular, a pair of filling measured foliations will determine a unique half-translation surface structure and consequently a unique singular flat quadratic differential metric.

\subsection{Minimal Lagrangians in $\mathbb{H}^{2} \times \mathbb{H}^{2}$}

A \textit{minimal Lagrangian} $\widetilde{\Sigma}$ in $\mathbb{H}^{2} \times \mathbb{H}^{2}$ is a minimal surface which is Lagrangian with respect to the symplectic form $\omega \oplus -\omega$, where $\omega$ is the standard K{\"a}hler form on $\mathbb{H}^{2}$. Any $\rho \in \mathrm{Max}(S)$ acts on $\mathbb{H}^{2} \times \mathbb{H}^{2}$, and Schoen \cite{Schoenharmonic} has shown to each such $\rho$, there is a unique $\rho$-equivariant minimal Lagrangian $\widetilde{\Sigma}_{\rho}$ in $\mathbb{H}^{2} \times \mathbb{H}^{2}$, thereby providing a geometric interpretation to representations in $\mathrm{Max}(S)$. The second author \cite{Charles_dPSL} has studied the degeneration of these minimal Lagrangians, and has shown that one may interpret the space $\mathbb{P}\mathrm{Core}(\mathcal T,\mathcal T)$ as the boundary of the maximal component $\mathrm{Max}(S)$.

\subsection{Induced metrics and projectivized mixed structures}

The induced metric on the unique $\rho$-equivariant minimal Lagrangian descends to a metric on $S$. It is not too difficult (see \cite[Proposition 4.2]{Charles_dPSL}) to see this metric is in fact negatively curved. Hence, by the result of \cite{Otal}, its marked length spectrum determines the metric. The marked length spectrum is the data of both the curve class and the length of its geodesic representative in the given homotopy class. Let Ind$(S)$ denote the space of induced metrics coming from the $\rho$-equivariant minimal Lagrangians. Then in fact one may embed Ind$(S)$ into the space of projectivized marked length spectra. Its closure is then determined to be precisely the space Ind$(S)$ together with the projectivized mixed structures (\cite[Theorem 5.5]{Charles_dPSL}).

\section{Core of product of trees}\label{sec:cores}

In this section we recall the notion of core of a product of trees and describe its geometry in the case of trees dual to measured laminations. The core of a product of two $\mathbb{R}$-trees can actually be defined for any pair of $\mathbb{R}$-trees each admitting a $\pi_{1}(S)$-action. It is not necessary that the $\mathbb{R}$-trees be dual to measured foliations. However, we will specifically mention when particular properties of cores are germane only to $\mathbb{R}$-trees dual to measured foliations. The main reference for the material covered here is \cite{Guirardel}. 

Given an $\mathbb{R}$-tree $T$, a \textit{direction} $\delta$ based at a point $p \in T$ is a connected component of $T \setminus \{p\}$. For a product $T_{1} \times T_{2}$ of $\mathbb{R}$-trees, a \textit{quadrant} $Q$ based at $(p_{1}, p_{2}) \in T_{1} \times T_{2}$ is a product $\delta_{1} \times \delta_{2}$ of directions. If the $\mathbb{R}$-trees $T_{1}, T_{2}$ are equipped with a $\pi_{1}(S)$-action by isometries, then we say a quadrant $Q$ is \textit{heavy} if there exists a sequence $\{\gamma_{n}\} \subset \pi_{1}(S)$ for which: 
\begin{enumerate}[label=\roman*)]
    \item $\gamma_{n} \cdot p_{i} \in \delta_{i} $, and 
    \item $d_{i}(\gamma_{n} \cdot p_{i}, p_{i}) \to \infty$ as $n \to \infty$ .
\end{enumerate}
Otherwise the quadrant is said to be \textit{light}. Following Guirardel \cite{Guirardel}, the \textit{core} $\mathcal{C}(T_{1}, T_{2})$ of $T_{1} \times T_{2}$ is 
\[
T_{1} \times T_{2}  \setminus \bigsqcup_{Q \,\,\text{light}} Q.
\]

When $T_{1}$ and $T_{2}$ are dual to measured laminations, the core $\mathcal{C}(T_{1}, T_{2})$ is always non-empty since the $\pi_1(S)$-actions are irreducible (\cite[Proposition 3.1]{Guirardel}).

However, even when $T_{1}$ and $T_{2}$ are dual to measured foliations, one pathology may still occur: $\mathcal{C}(T_{1}, T_{2})$ may be disconnected. This happens, for instance, when $T_{1}=T_{2}$ and $T_{1}$ is dual to a multicurve. However, in such cases, Guirardel introduced a canonical way of extending the core to a connected subset of $T_{1}\times T_{2}$ with convex fibers. With abuse of terminology, we will still refer to this canonical extension as the core of $T_{1}\times T_{2}$. The following result completely characterizes when this extension needs to be considered.

\begin{deff}\label{def:refinement} Given two real trees $T$ and $T'$ endowed with an action of $\pi_{1}(S)$, we say that $T$ is a {\em refinement} of $T'$ if there is an equivariant map $f:T \rightarrow T'$ such that for all $x,y,z \in T$ if $z$ lies in the geodesic $[x,y]$ connecting $x$ and $y$, then $f(z)$ belongs to $[f(x), f(y)]$.
\end{deff}

\begin{prop}[{\cite[Proposition 4.14]{Guirardel}}]\label{prop:augmented} Let $T_{1}$ and $T_{2}$ be trees dual to measured laminations. Then the core $\mathcal{C}(T_{1}, T_{2})$ is disconnected if and only if $T_{1}$ and $T_{2}$ are refinements of a common non-trivial simplicial tree $T$.
\end{prop}

For example the assumptions of Proposition \ref{prop:augmented} are satisfied if $T_{1}$ and $T_{2}$ are dual to measured laminations $\lambda_{1}$ and $\lambda_{2}$ with common isolated leaves. 

When $T_{1}$ and $T_{2}$ are both dual to measured laminations $\lambda_{1}$ and $\lambda_{2}$, we can actually realize the core $\mathcal{C}(T_{1}, T_{2})$ more concretely. Before describing this construction, we need the following result, which can be seen as a special case of the decomposition theorem for general geodesic currents in (\cite{BIPP}, see also \cite{BIPPmetric}) about how two measured laminations interact on subsurfaces.

\begin{lemma}\label{lm:decomposition_pair_laminations} Let $\lambda_{1}$ and $\lambda_{2}$ be measured laminations on $S$. Then there is a system of non-trivial, pairwise non-homotopic, disjoint, simple closed curves $\gamma_{1}, \dots, \gamma_{n}$ such that on each connected component $S'$ of $S\setminus\cup_{j} \gamma_{j}$ either
\begin{enumerate}[label=\roman*)]
    \item $\lambda_{1}+\lambda_{2}$ is a (possibly zero) measured lamination on $S'$; or
    \item $\lambda_{1}$ and $\lambda_{2}$ are transverse and fill $S'$, i.e. for all measured lamination $\nu$ on $S'$ we have $i(\lambda_{1}, \nu)+i(\lambda_{2},\nu)\neq 0$.
\end{enumerate}
\end{lemma}
\begin{proof} Consider a maximal collection of non-trivial, pairwise non-homotopic, disjoint, simple closed curves $\gamma_{j}$ such that
\[
    i(\lambda_{1}, \gamma_{j})+i(\lambda_{2}, \gamma_{j})=0.
\]
We claim that this collection of curves satisfies the requirement of the lemma. Indeed, let $S'$ be a connected component of $S\setminus \cup_{j}\gamma_{j}$. We need to show that if the pair $(\lambda_{1}, \lambda_{2})$ does  not fill the subsurface $S'$, then $\lambda_{1}+\lambda_{2}$ is a lamination on $S'$, or, equivalently, $\lambda_{1}$ and $\lambda_{2}$ are nowhere transverse on $S'$. The claim is clearly true if the support of either $\lambda_{1}$ or $\lambda_{2}$ does not intersect $S'$, so we can assume that both have support on $S'$. Because the pair $(\lambda_{1}, \lambda_{2})$ does not fill $S'$ by assumption, there is a measured lamination $\nu$ on $S'$ such that $i(\lambda_{1}, \nu)+i(\lambda_{2}, \nu)=0$.
On the other hand, by hypothesis, $i(\lambda_{1}, \gamma)+i(\lambda_{2}, \gamma)\neq 0$ for all non-peripheral simple closed curves $\gamma$ on $S'$. Therefore, the measured lamination $\nu$ does not contain isolated closed leaves. 
Let us first consider the case in which $\nu$ fills the subsurface $S'$, in the sense that the complement of $\nu$ only consists of disks and annuli. We note that then necessarily the support of $\nu$ must contain the support of $\lambda_{1}$ and $\lambda_{2}$ because otherwise $\lambda_{1}$ and $\lambda_{2}$ would intersect $\nu$ transversely somewhere. But this implies that $\lambda_{1}$ and $\lambda_{2}$ are nowhere transverse, being both contained in the support of a measured lamination. 
We now reduce the general case to this, by showing that $\nu$ must fill $S'$. Assume the opposite, and let $S''\subset S'$ be a subsurface filled by $\nu$. Note that at least one between $\lambda_{1}$ and $\lambda_{2}$ intersects the boundaries of $S''$ transversely. Without loss of generality we assume it is $\lambda_{1}$. Since $\nu$ fills $S''$, the support of $\lambda_{1}$ intersects $\nu$ transversely, but this contradicts the fact that $i(\nu, \lambda_{1})=0$. 
\end{proof}

The last ingredient we need is an explicit realization of a tree $T_{\lambda}$ dual to a measured lamination $\lambda$. The construction goes as follows (see \cite{MO_growth} for more details). Fix an auxiliary hyperbolic metric on $S$ and identify $\widetilde{S}$ with $\h^{2}$. Let $\widetilde{\lambda}$ be the lift of $\lambda$ under the covering map $\pi: \h^{2} \rightarrow S$. We define the metric space $\mathrm{pre}(T_{\lambda})$ where points of $\mathrm{pre}(T_{\lambda})$ are the connected components of $\h^{2} \setminus \widetilde{\lambda}$ and the distance is computed as follows: if $x,y \in \mathrm{pre}(T_{\lambda})$ correspond to connected components $C_{x},C_{y}$ of $\h^{2} \setminus \widetilde{\lambda}$ then 
\[
    d_{\lambda}(x,y)=\inf\left\{ \int_{\gamma} d\widetilde{\lambda} \ \Big| \ \gamma:[0,1] \rightarrow \h^{2}, \ \gamma(0) \in C_{x}, \ \gamma(1) \in C_{y}\right\}
\] 
The tree $T_{\lambda}$ is then the unique $\R$-tree that contains $\mathrm{pre}(T_{\lambda})$ such that any point of $T_{\lambda}$ lies in a segment with vertices in $\mathrm{pre}(T_{\lambda})$. Note that we have a natural projection map $p_{\lambda}: \h^{2} \setminus \widetilde{\lambda} \rightarrow T_{\lambda}$. If $\lambda$ has no isolated leaves, this map extends continuously to a map, still denoted by $p_{\lambda}$, defined on the entire $\h^{2}$. Otherwise, the continuous extension is obtained by first replacing each isolated leaf $\ell$ in $\widetilde{\lambda}$ with a strip $\ell \times [-\epsilon, \epsilon]$ endowed with a uniform measure with total mass equal to $\widetilde{\lambda}_{|_{\ell}}$. 

There is also another way of realizing the tree dual to a measured lamination using the language of measured foliations. Let $\mathcal{F}$ denote the measured foliation corresponding to the measured lamination $\lambda$ under the homeomorphism between $\mathcal{MF}(S)$ and $\mathcal{ML}(S)$. Let $\widetilde{\mathcal{F}}$ be its lift to $\h^{2}$. Then the tree $T_{\lambda}$ can be defined as the quotient $\h^{2}/ \sim$, where $\sim$ denotes the equivalence relation
\[
    x \sim y \Leftrightarrow d_{\mathcal{F}}(x,y)=0
\]
and
\[
    d_{\mathcal{F}}(x,y)=\inf\{ i(\widetilde{\mathcal{F}}, \gamma) \ | \ \gamma:[0,1] \rightarrow \h^{2}, \ \gamma(0)=x, \ \gamma(1)=y\}.
\]
More concretely, $T_{\lambda}$ identifies with the leaf space of $\widetilde{\mathcal{F}}$ with distance given by integrating the measure of $\widetilde{\mathcal{F}}$ along arcs transverse to the leaves. We denote by $\pi_{\lambda}$ the natural projection $\pi_{\lambda}:\h^{2} \rightarrow T_{\lambda}$. 

We are now ready to describe the core of a product of two trees $T_{1}$ and $T_{2}$ dual to measured laminations $\lambda_{1}$ and $\lambda_{2}$ on $S$.  Lemma \ref{lm:decomposition_pair_laminations} furnishes a decomposition of $S$ into subsurfaces that we lift to a decomposition of $\h^{2}$. The regions of this decomposition come in two flavors according to whether they project to subsurfaces where $\lambda_{1}+\lambda_{2}$ is a lamination or to subsurfaces where the pair $(\lambda_{1}, \lambda_{2})$ fills. Following the statement of Lemma \ref{lm:decomposition_pair_laminations}, we call these regions of type $i)$ and type $ii)$, respectively. On the regions $\Omega \subset \h^{2}$ of type $i)$, the union of the lifts $\widetilde{\lambda}_{1}$ and $\widetilde{\lambda}_{2}$ can be regarded as the lift of the measured lamination $\lambda_{0}=\lambda_{1}+\lambda_{2}$. We denote by $T_{0}$ the tree dual to $\lambda_{0}$. Note that, for each region $\Omega \subset \h^{2}$ of type $i)$, we have a map $p_{0}:=p_{\lambda_{0}}$, as defined before, and two natural collapsing maps $c_{j}: T_{0} \rightarrow T_{j}$ for $j=1,2$. On the regions of type $ii)$, we replace the measured laminations $\widetilde{\lambda}_{i}$ with the corresponding measured foliations $\widetilde{\mathcal{F}}_{i}$
and consider the projections $\pi_{i}:=\pi_{\lambda_{i}}$ as described previously. Following Guirardel (\cite[Example 4]{Guirardel}, \cite[Proposition 6.1]{Guirardel}), the core $\mathcal{C}(T_{1}, T_{2})$ is the image of the map $F:\widetilde{S}\rightarrow T_{1}\times T_{2}$ defined as follows:
\begin{equation}\label{eq:F}
    F(x) = \begin{cases}
            (\pi_{1}\times \pi_{2})(x) \hspace{2cm} \text{if $x$ belongs to a region of type $ii)$} \\
            (c_{1}\times c_{2})({p_{0}}(x)) \hspace{1.45cm}  \text{if $x$ belongs to a region of type $i)$} \ . 
    \end{cases} 
\end{equation}
Note that $F$ is well-defined and continuous on the boundary $\widetilde{\gamma}$ between two different regions of $\widetilde{S}$ because $\widetilde{\gamma}$ is the lift of a curve $\gamma_{j}$ given by Lemma \ref{lm:decomposition_pair_laminations} which, by definition, has vanishing intersection number with $\lambda_{0}$, $\mathcal{F}_{1}$, and $\mathcal{F}_{2}$, hence $(\pi_{1}\times \pi_{2})(\widetilde{\gamma})$ and $(c_{1}\times c_{2}) ({p_{0}}(\widetilde{\gamma}))$ is a single point. 

It follows from this explicit description of $\mathcal{C}(T_{1},T_{2})$ that the core is, in general, a 2-dimensional subcomplex of $T_{1}\times T_{2}$ that is invariant under the diagonal action of $\pi_{1}(S)$. Moreover, the $2$-dimensional pieces of $\mathcal{C}(T_{1},T_{2})$ are exactly the images of regions of type $ii)$ and are foliated by two families of transverse foliations. Their quotients under the group action are the union of the subsurfaces of $S$ in which $\lambda_{1}$ and $\lambda_{2}$ fill, endowed with the foliations $\mathcal{F}_{1}$ and $\mathcal{F}_{2}$ (\cite[Example 4]{Guirardel}). In particular, the 2-dimensional pieces of $\mathcal{C}(T_{1},T_{2})$ are the universal covers of half-translation surfaces. On the other hand, the images under $F$ of regions $\Omega$ of type $i)$ are $1$-dimensional subcomplexes of $T_{1}\times T_{2}$. Each such $\Omega$ can be seen as the universal cover of a subsurface $S'$ of $S$ where the restriction of $\lambda_{1}+\lambda_{2}$ is a measured lamination. Let $T_{1}'\subset T_{1}$ and $T_{2}'\subset T_{2}$ be the corresponding subtrees. It turns out (\cite[Section 6]{Guirardel}) that $F(\Omega)$ is an $\R$-tree that is a common refinement of $T_{1}'$ and $T_{2}'$ if endowed with the distance
\[
    d_{0}(x,y)=d_{1}(x_{1},y_{1})+d_{2}(x_{2},y_{2}) \ \ \ \ \text{$x=(x_{1},x_{2}), y=(y_{1},y_{2}) \in T_{1}'\times T_{2}'$}
\]
where $d_{j}$ denotes the distance on $T_{j}$. 

\begin{lemma}\label{lm:1-dim core} The $\R$-tree $(F(\Omega), d_{0})$ is isometric to the tree dual to the measured lamination $\lambda_{0}=\lambda_{1}+\lambda_{2}$ restricted to $S'$.
\end{lemma}
\begin{proof} The tree $F(\Omega)$ inherits from $T_{1}'\times T_{2}'$ an isometric action of $\pi_{1}(S')$. We can define a length function
\begin{align*}
    \ell: \pi_{1}(S') &\rightarrow \R^{+} \\
            \gamma &\mapsto \lim_{n \to +\infty}\frac{1}{n}d_{0}(x, \gamma^{n}\cdot x)
\end{align*}
where $x$ is any point in $F(\Omega)$ (the definition is independent of the choice of $x$). Since the action of $\pi_{1}(S')$ on $F(\Omega)$ is minimal and irreducible, by \cite[Theorem A.5]{GL_JSJ}, the isometry class of $(F(\Omega), d_{0})$ is completely determined by its length function. However, it is clear from the definition of $\ell$ and $d_{0}$ that $\ell=\ell_{0}:=\ell_{1}+\ell_{2}$, where $\ell_{j}$ denotes the analogously defined length functions on $T_{1}'$ and $T_{2}'$. On the other hand, $\ell_{0}$ is exactly the length function of the tree dual to the measured lamination $\lambda_{0}$, and the claim follows.
\end{proof}

The ambient space $T_{1}\times T_{2}$ has, however, another natural distance defined by
\[
    d(x,y)=\sqrt{d_{1}(x_{1},y_{1})^{2}+d_{2}(x_{2},y_{2})^{2}} \ \ \ \ \text{$x=(x_{1},x_{2}), y=(y_{1},y_{2}) \in T_{1}\times T_{2}$} \ . 
\]
This induces a path metric $d_\mathcal{C}$ on the core $\mathcal{C}$ of $T_{1}\times T_{2}$, where the $d_{\mathcal{C}}$-distance between two points in the core is the infimum of the length of all paths connecting the points and entirely contained in the core, where the length is computed using the distance $d$. Guirardel showed (\cite[Proposition 4.9]{Guirardel}) that the core is a CAT(0)-space if endowed with this path distance $d_{\mathcal{C}}$. In particular, since $F(\Omega)$ does not contain topological circles by Lemma \ref{lm:1-dim core}, we can conclude that $F(\Omega)$ endowed with the restriction of $d_{\mathcal{C}}$ is still an $\R$-tree. 

We will denote by $\Core(\mathcal{T}, \mathcal{T})$ the space of cores of the product of two trees dual to measured laminations on $S$ endowed with this path distance.

\begin{prop} \label{prop:core=productlaminations} $\Core(\mathcal{T}, \mathcal{T})$ is homeomorphic to $\mathcal{ML}(S) \times \mathcal{ML}(S)$.
\end{prop}
\begin{proof} Since the core of a product of trees is uniquely determined by the two factors, the result follows immediately from the homeomorphism between the space of trees dual to measured laminations and $\mathcal{ML}(S)$.
\end{proof}

We note that there is a natural $\R^{+}$-action on $\Core(\mathcal{T}, \mathcal{T})$ given by rescaling the induced metric on the core, which, under the homeomorphism above, corresponds to the diagonal action of $\R^{+}$ by scalar multiplication on the measures. We denote by $\Pp\Core(\mathcal{T}, \mathcal{T})$ the quotient $\Core(\mathcal{T}, \mathcal{T})/\R^{+}$. It follows that $\Pp\Core(\mathcal{T}, \mathcal{T})$ is homeomorphic to $\Pp(\mathcal{ML}(S) \times \mathcal{ML}(S))$. In particular, it is topologically a sphere of dimension $12g-13$.

\section{Thurston's compactification}\label{sec:Thurston}

Recall that we denote by $\text{Max}(S)$
the space of conjugacy classes of representations $\rho=(\rho_1,\rho_2)$ of the fundamental group of a closed connected oriented surface $S$ of negative Euler characteristic into the Lie group $\dPSL$ such that $e(\rho_1)+e(\rho_2)=4g-4$. Here, $e$ denotes the Euler number of the representation. It follows from \cite{Goldman_topcomp} that $\rho_1$ and $\rho_2$ are both Fuchsian representations. Therefore,  as $\text{Max}(S)$ may be thought of as the product of two copies of Teichm{\"u}ller space, it is homeomorphic to an open cell of dimension $12g-12$.

The main goal of this section is to prove Theorem \ref{thm:closedballIntro} from the introduction, which we restate below for the convenience of the reader. 
\begin{thm}\label{thm:closedball} The disjoint union
\[
\mathfrak B=\mathrm{Max}(S)\sqcup \mathbb{P}\mathrm{Core}(\mathcal T,\mathcal T)
\]
is homeomorphic to a closed ball of dimension $12g-12$.

\end{thm}

We begin by recalling the topology placed on $\mathfrak B$. The maximal component Max$(S)$ is naturally homeomorphic to the product of two copies of Teichm{\"u}ller space. This in turn, by the result of Schoen is homeomorphic to the space of equivariant minimal Lagrangians in $\mathbb{H}^{2} \times \mathbb{H}^{2}$. Under the Gromov-Hausdorff topology, diverging sequences of minimal Lagrangians subconverge to the (projective) core of a product of two trees (\cite[Theorem 8.1]{Charles_dPSL}). The two trees are dual to a pair of measured laminations, and the topology on $\mathfrak B$ is compatible with the Thurston compactification on Teich$(S) \times$ Teich$(S)$ in the following way: if $(\rho_{1,n}, \rho_{2,n}) \to [\lambda_{1}, \lambda_{2}]$, then the associated minimal Lagrangians converge to the core of $T_{1} \times T_{2}$, where $T_{i}$ is dual to $\lambda_{i}$.

% \GM{Should this paragraph be part of the proof? There is some overlap between this paragraph and the first paragraph of the proof}
Fix a complex structure $J$ on $S$ and denote by $X$ the Riemann surface $(S,J)$. Then for any hyperbolic metric $h \in \text{Teich}(S)$  there is \cite{eells} a unique \cite{hartman1967homotopic} harmonic map $w_h \colon X \to (S,h)$ in the homotopy class of the identity. Harmonicity of $w_h$ ensures that the Hopf differential $q_h=\left(w_h^{*}h\right)^{(2,0)}$ is a holomorphic quadratic differential on $X$. The vector space $\text{QD}(X)$ of holomorphic quadratic differentials on $X$ has a natural norm given by the $L^{2}$-norm with respect to the uniformizing hyperbolic metric $\sigma$ of $X$. With an abuse of notation, we will still denote by $X$ the hyperbolic surface $(S,\sigma)$. The map which assigns to a point in Teichm{\"u}ller space its corresponding Hopf differential is a homeomorphism \cite{Wolf_harmonic}. 

\begin{proof}[Proof of Theorem \ref{thm:closedball}] 
By Theorem 6.13 of \cite{Charles_dPSL}, the space $\text{Max}(S) \sqcup \Pp\mathrm{Core}(\mathcal T,\mathcal T)$ is naturally homeomorphic to $\text{Teich}(S) \times \text{Teich}(S) \sqcup \mathbb{P}(\mathcal{ML}(S) \times \mathcal{ML}(S))$, so it suffices to prove the latter is homeomorphic to a closed ball of dimension $12g-12$.

As $\mathbb P(\mathcal{ML}(S) \times \mathcal{ML}(S))$ is homeomorphic to a sphere of dimension $12g-13$, the remainder of the proof consists of describing how to attach this topological space to the open cell $\text{Teich}(S) \times \text{Teich}(S)$ to obtain a closed ball.

We start by fixing a complex structure $J$ on $S$. Let $X=(S,J)$ be the resulting Riemann surface. By the Wolf parameterization \cite{Wolf_harmonic}
\[
\text{Teich}(S) \times \text{Teich}(S) \cong \mathrm{QD}(X)\oplus \mathrm{QD}(X)
\] 
via the map $\Phi(\rho_{1}, \rho_{2}) = (q_{\rho_{1}}, q_{\rho_{2}})$.
We equip $\mathrm{QD}(X)\oplus \mathrm{QD}(X)$ with the norm
\[                  
\|q\|=\max\left(\|q_1\|,\|q_2\|\right),                          
\]
and consider                                
\[
\mathrm{BPQD}(X)=\{q=(q_1,q_2)\colon \|q\|<1\}.
\]
which is, topologically, a ball of dimension $12g-12$.
We will need the following lemma.
\begin{lemma}\label{lemma:inj} The map
\begin{align*}
\beta\colon  \mathrm{QD}(X)\oplus \mathrm{QD}(X)&\to \mathrm{BPQD}(X)\\
q=(q_1,q_2)&\mapsto \frac{4q}{1+4\|q\|}
\end{align*}
is continuous, injective, and proper. Hence $\beta$ is a homeomorphism. 
\end{lemma}
\begin{proof}
Suppose $\beta(q_{1},q_{2}) = \beta(\phi_{1}, \phi_{2})$. It follows then that $q_{1} = k \phi_{1}$ and $q_{2} = k \phi_{2}$ for some $k \in \mathbb{R}$. Writing out  $\beta(q_{1},q_{2}) = \beta(q_{1}/k, q_{2}/k)$, basic algebra shows $k=1$. Continuity and properness follow by inspection.
\end{proof}
%We have a homeomorphism
%\[
%i\colon  \mathrm{QD}(\sigma)\oplus \mathrm{QD}(\sigma)\to \text{MF}\times\text{MF}\cong \mathcal T\times\mathcal T
%\]
%where $i$ assigns to the pair of holomorphic quadratic differentials the pair given by their horizontal foliations.

%Now, let $\rho_n\in\text{Max}(S,\dPSL)$ denote a diverging sequence, then there exists $\phi_n\in\mathrm{QD}(\sigma)\oplus \mathrm{QD}(\sigma)$ with $\rho_n=\beta(\phi_n)$, $\|\beta(\phi_n) \|<1$ and
%\[
%\lim_n \|\beta(\phi_n)\|=1.
%\]
%In particular, there exists $\phi_\infty\in\mathrm{SPQD}(\sigma)$ such that $\lim_n\beta(\phi_n)=\phi_\infty$.
%\AT{SPQD hasn't been defined}

%Then, $i(\beta(\phi_n))\to i(\phi_\infty)$ as $n$ goes to $\infty$.

We will now describe the attaching map. Consider the map 
% \AT{Before defining the map, we need to describe the topology of the domain.}
\[
\psi: \text{Teich}(S) \times \text{Teich}(S) \sqcup \mathbb{P}(\mathcal{MF}(S) \times \mathcal{MF}(S)) \to \overline{\text{BPQD}(X)}
\]
defined by 
\[
    \psi(x)=\begin{cases}
             \ \ \beta(\Phi(x))  \ \ \   \ \ \ \ \ \  \ \ \ \ \ \text{if }x\in \text{Teich}(S)\times \text{Teich}(S)\\
            \lim\limits_{n \to + \infty} \beta(\Phi(x_{n})) \ \ \ \ \ \ \text{if } x\in \mathbb{P}(\mathcal{MF}(S) \times \mathcal{MF}(S))\text{ and } x_n\to x
            \end{cases}
\]
%\[ 
%\psi(x)=
%\begin{cases}
%\left(\frac{\Phi(x)}{\|\Phi(x)\|},\frac{4\|\Phi(x)\|}{1+4\|\Phi(x)\|}\right)&\text{ \ \ if }x\in \text{Teich}(S)\times \text{Teich}(S)\\
%\left(\lim_n\frac{\Phi(x_n)}{\|\Phi(x_n)\|},1\right)&\text{ \ \ if }x\in \mathbb{P}(\mathcal{MF}(S) \times \mathcal{MF}(S))\text{ and } x_n\to x
%\end{cases}
%\]
%Here, we have used polar coordinates $(\theta, r)$ for $\overline{\text{BPQD}(X)}$, where $\theta \in \text{SPQD}(X) = \{\phi=(\phi_1,\phi_2)\colon \|\phi\|=1\}$ and $r \in [0,1]$. 
We show first that the map $\psi$ is well-defined. Suppose $x_{n} =(X_{1,n}, X_{2,n}) \to x$ and $x'_{n} =(X'_{1,n}, X'_{2,n}) \to x$, where $x = [\lambda_{1}, \lambda_{2}] \in \mathbb{P}(\mathcal{MF}(S) \times \mathcal{MF}(S))$. That is to say, there exists sequences of real numbers $c_{n}, d_{n}$ for which the rescaled hyperbolic surfaces $\widetilde{X}_{i,n}/c_{n}$ and $\widetilde{X'}_{i,n}/d_{n}$ converge to $\mathbb{R}$-trees $T_{i}, T'_{i}$ dual to laminations $\lambda_{i}$ and $\lambda_{i}'$ such that $[(\lambda_{1}, \lambda_{2})]=[(\lambda_{1}',\lambda_{2}')]$. By \cite{Wolf_harmonic}, the sequences $c_{n}$ and $d_{n}$ can be taken to be $\|\Phi(x_{n})\|$ and $\| \Phi(x_{n}')\|$. Note that, a priori, $\lambda_{i}'=k\lambda_{i}$ for some $k>0$. With such rescaling, the harmonic maps $h_{i,n}: \widetilde{X} \to \widetilde{X}_{i,n}/c_{n}$ converge to the harmonic map $h_{i}: \widetilde{X} \to T_{i}$ given by projection onto the leaf space of the measured foliation $\widetilde{\mathcal{F}}_{i}$ corresponding to $\widetilde{\lambda}_{i}$ \cite[Corollary 5.2]{wolf1995harmonic}. Moreover the sequence of Hopf differentials $q_{i,n}$ of $h_{i,n}$ converges to the Hopf differential $q_{i}$ of $h_{i}$ (here take the quotient so that $q_{i}$ is a holomorphic quadratic differential on $X$ and not $\widetilde{X}$). Finally, the differential $q_{i}$ is the unique holomorphic quadratic differential on $X$ whose horizontal foliation is Whitehead equivalent to $\mathcal{F}_{i}$. Likewise the sequence of harmonic maps $h'_{i,n}: \widetilde{X} \to \widetilde{X'}_{i,n}/{d_{n}}$ converges to the harmonic map $h': \widetilde{X} \to T'_{i}$, whose Hopf differential is $q'_{i}$ is the limit of the Hopf differentials $q_{i,n}'$ of $h'_{i,n}$ and has horizontal foliation $\mathcal{F}'_{i}$ corresponding to the lamination $\lambda_{i}'$. Notice, in addition, that $(q_{1,n},q_{2,n})=\Phi(x_{n})/\| \Phi(x_{n})\|$ and similarly  $(q_{1,n}',q_{2,n}')=\Phi(x_{n}')/ \|\Phi(x_{n}')\|$. It follows that the limits of $\beta(\Phi(x_{n}))$ and $\beta(\Phi(x_{n}'))$ as $n \to +\infty$ exist and coincide with $(q_{1}, q_{2})$ and $(q_{1}', q_{2}')$. As the distance functions $d_{i}$ and $d_{i}'$ on $T_{i}$ and $T_{i}'$ satisfy $d_{i} = k \cdot d_{i}'$, then, by homogeneity of the Hopf differential, one has $q_{i} = k \cdot q'_{i}$. Since the pairs $(q_{1}, q_{2})$ and $(q_{1}', q_{2}')$ both have unit norm, we conclude that $k=1$ and the limits of $\beta(\Phi(x_{n}))$ and $\beta(\Phi(x_{n}'))$ as $n \to +\infty$ are equal. 

Continuity follows almost immediately: the map $\beta\circ\Phi$ is continuous on the interior and extends continuously to the boundary by a diagonal argument. Indeed, we can approximate a sequence along the boundary by sequences in the interior.
% \GM{Can we say more here? I am not sure what diagonal argument we are referring to.}

Bijectivity of $\psi$ on the interior also follows by \cite{Wolf_harmonic} and Lemma \ref{lemma:inj}. On the boundary, given $q=(q_{1}, q_{2})$ with $\|q\|=1$, if $X_{i,t}$ is the hyperbolic surface corresponding to the rays $tq_{i}$ in Wolf's parameterization of $\mathrm{Teich}(S)$, we have that $\beta(\Phi(X_{1,t}, X_{2,t})) \to q$ as $t\to \infty$, thus $\psi$ is surjective on the boundary. Since the limit of $\beta(\Phi(x_{n}))$ along diverging sequences in $x_{n}\in \mathrm{Teich}(S)\times\mathrm{Teich}(S)$ only depends on the projective class of the limit of $x_{n}$ and not on the particular sequence, we deduce that $\psi$ is injective on the boundary, because every point in $\Pp(\mathcal{ML}(S)\times\mathcal{ML}(S))$ can be obtained as a limit along a ray defined above and the limit of $\beta\circ \Phi$ along distinct rays is different.

It remains to prove $\psi^{-1}$ is continuous. We can actually write the inverse explicitly:
\[
    \psi^{-1}(q_{1},q_{2})= \begin{cases}
                    \Phi^{-1}(\beta^{-1}(q_{1}, q_{2})) \ \ \ \ \ \ \ \ \ \ \text{if $\|(q_{1},q_{2})\|<1$} \\
                    [\lambda_{1}, \lambda_{2}] \ \ \ \ \ \ \ \ \ \ \ \ \ \ \ \ \ \ \ \ \ \ \text{if $\|(q_{1},q_{2})\|=1$} 
                    \end{cases}
\]
where $\lambda_{i}$ is the measured lamination corresponding to the horizontal foliation of $q_{i}$. Continuity of $\psi^{-1}$ on $\mathrm{BPQD}(X)$ is then a consequence of Lemma \ref{lemma:inj} and Wolf's parameterization. Continuity on the boundary follows from Hubbard-Masur theorem \cite{hubbard1979quadratic}. In general, if $q_{n}=(q_{1,n}, q_{2,n}) \in \mathrm{BPQD}(X)$ converges to $(q_{1}, q_{2}) \in \partial\overline{\mathrm{BPQD}(X)}$, then there is a sequence of scaling factors $c_{n}$ such that the pair of hyperbolic surfaces $x_{n}=\psi^{-1}(q_{1,n}, q_{2,n})$ rescaled by $c_{n}$ converge to real trees $T_{1}, T_{2}$ dual to measured laminations $\lambda_{1}, \lambda_{2}$. We need to show that $\psi^{-1}(q_{1},q_{2})$ is equal to $[\lambda_{1}, \lambda_{2}]$. Assume not, then we would have, by injectivity and continuity of $\psi$, 
\begin{align*}
    (q_{1},q_{2})& = \psi(\psi^{-1}(q_{1},q_{2})) \neq \psi([\lambda_{1}, \lambda_{2}]) = \lim_{n\to +\infty}\psi(x_{n})=\lim_{n\to +\infty}(q_{1,n}, q_{2,n})
\end{align*}
which contradicts the assumption on $(q_{1,n}, q_{2,n})$. 

Finally, we remark the compactification in \cite{Charles_dPSL} is independent of the choice of a base point, so that the role of the base point $(S,J)$ is merely an auxiliary one. This completes the proof of the theorem.
% \AT{Need to add that the boundary is independent of the choice of the base point using Charles' thesis}
\end{proof}

% It is clear that the action of the mapping class group extends to the compactification. As the compactified space is a closed ball, one may apply Brouwer's fixed point theorem.\GM{We might comment this - It overlaps with what's in section 4}

\section{Fixed point for the mapping class group action}

In this section, we study the action of the mapping class group $\mathrm{MCG}(S)$ on the compactification $\mathfrak B=\overline{\mathrm{Max}(S)}$ constructed in Theorem \ref{thm:closedball}. We wish to study the fixed points of this action. We will need the following observations.
\begin{lemma}\label{lem:MCGextends} The action of the mapping class group on $\mathrm{Max}(S)$ extends continuously to the closure $\mathfrak B=\overline{\text{Max}(S)}$. 
% \AT{If in the previous section we show that the topology is simply the topology of the length spectrum, this lemma should be clear}
\end{lemma}
% \begin{proof} Let $\phi\in\mathrm{MCG}(S)$ and let $x_n=(x_n^1,x_n^2)$ be a sequence in $\mathfrak B$.  We want to show that if $x_n\to x=(x^1,x^2)\in\mathfrak B$ then $\phi(x_n)\to \phi(x)$. 

% The mapping class group of $S$ acts on $\text{Max}(S)$ by $\phi(x^1,x^2)=(\phi(x^1),\phi(x^2))$. This action is continuous, as it is continuous in each factor.

% Suppose $x=(x^1,x^2)\in\partial \mathfrak B$. The action of $\mathrm{MCG}(S)$ extends continuously to the Thurston boundary, so $\phi(x^1_n)\to\phi(x^1)$, $\phi(x^2_n)\to\phi(x^2)$, and thus $\phi(x_n)\to\phi(x)$.
% \end{proof}

\begin{cor} For every $\phi\in\mathrm{MCG}(S)$, there exists $x\in\mathfrak B$ such that $\phi(x)=x$.
\end{cor}

The first main goal of this section is to analyze these fixed points via the celebrated Nielsen-Thurston classification, which we recall for future reference.

\begin{thm}[{Nielsen-Thurston classification, see \cite[Chapter 13]{farb2011primer}}]\label{thm:thurstonclass} Any diffeomorphism $\phi$ on $S$ is isotopic to a map $\phi'$ satisfying one of the following mutually exclusive conditions:
\begin{enumerate}
    \item{\em(periodic)} $\phi'$ is of finite order;
    \item{\em(reducible)} $\phi'$ is not periodic, and there is a nonempty set $\{c_1,\dots, c_r\}$ of isotopy classes of essential pairwise disjoint simple closed curves in $S$ such that $\{\phi'(c_i)\}_{i=1}^r=\{c_i\}_{i=1}^r$;
    \item{\em(pseudo-Anosov/pA)} there exists $\lambda>1$ and two transverse measured foliations $\mathcal F$ and $\mathcal F'$ such that
    \[
    \phi'(\mathcal F)=\lambda\mathcal F\text{ and } \phi'(\mathcal F')=\frac{1}{\lambda}\mathcal F'.
    \]
\end{enumerate}
\end{thm}
\begin{rmk} Note that our definition of reducible mapping class is non-standard as we assume that if $\phi$ is reducible, then it is not periodic. We do so to improve our exposition. The set $\{c_1,\dots, c_r\}$ in item (2) is a {\em reduction system} of $\phi$. The {\em canonical reduction system} $\{\bar c_1,\ldots,\bar c_k\}$ of $\phi$ reducible is the intersection of all the maximal (with respect to inclusion) reduction systems. Equivalently, each $\bar c_j$ is part of a reduction system and if $i(\bar c_j,c)\neq 0$ and $n\neq 0$, then $\phi^n(c)\neq c$.
\end{rmk}

We are ready to characterize the fixed points of a mapping class acting on $\mathfrak B$ and establish Proposition \ref{prop:upstairsfixedpts} from the introduction.

\begin{prop} Suppose $\phi\in\mathrm{MCG}(S)$ and $\phi(x)=x$ for some $x=(x_1,x_2)\in\mathfrak B$.
\begin{enumerate}
    \item If $\phi$ is periodic, then $x_1$ and $x_2$ are any two points fixed by $\phi$ in the Thurston compactification of Teichm\"uller space such that $(x_1,x_2)\in\mathfrak B$.
    \item If $\phi$ is pA, then $(x_1,x_2)\in\partial \mathfrak B$ and $x_1=0$, or $x_2=0$ or $x_1=x_2$.
\end{enumerate}
\end{prop}
 \begin{proof} 
\begin{enumerate} 
 \item 
 %If $x=(x_1,x_2)\in\mathrm{Max}(S)$, then $\phi$ fixes $x_i\in \text{Teich}(S)$. \AT{I don't think this last sentence is necessary} 
 If $\phi$ fixes $x$ projectively, there exists $\alpha>0$ such that $\phi(x_1,x_2)=(\alpha x_1,\alpha x_2)$. Since $\phi$ is periodic, we can check that $\alpha=1$. 

\item Since $\phi$ fixes the projective class of $(x_1,x_2)$, there exists $\alpha>0$ such that $\phi(x_1,x_2)=(\alpha x_1,\alpha x_2)$. On the other hand, since $\phi$ is pseudo-Anosov, there exist two measured laminations $y_1$ and $y_2$ and $\lambda>1$ such that $\phi(y_1)=\lambda y_1$ and $\phi(y_2)=\frac{1}{\lambda} y_2$. Since $\phi$ does not fix any other projective class of measured laminations \cite[Corollary 12.4]{fathi2021thurston}, it follows that $x_i=0$, $y_1$ or $y_2$ for $i=1,2$. We claim that $x\neq (y_1,y_2)$ (and,  symmetrically, $x\neq (y_2,y_1)$). Otherwise, because $i(y_1,y_2)\neq 0$
\[
\lambda\cdot i(y_1,y_2)=i(\phi(y_1),y_2)=\alpha\cdot i(y_1,y_2)= i(y_1,\phi(y_2))=\frac{1}{\lambda}\cdot i(y_1,y_2)
\]
which is a contradiction. \qedhere
 \end{enumerate}
 \end{proof}
% Let us assume $x=(x_1,x_2)\in\partial\mathfrak B$. Since $\phi$ needs to fix projectively the measured laminations $x_1$ and $x_2$, the statement follows by proving that $x_1$ and $x_2$ are nowhere transverse. 
% {\color{magenta} Since we are assuming $\phi$ is pA,
% there exists $\lambda>0$ and $\lambda\neq 1$ such that if $c$ is a multi-curve, then
% \[
% i(c,x_1)+i(c,x_2)=\lambda i(c,x_1)+\frac{1}{\lambda}i(c,x_2)\Rightarrow i(c,x_2)=\lambda i(c,x_1) 
% \]}
% Consider a sequence $c_n$ of multi-curves which approximates $x_1$. Then, up to possibly passing to subsequences,
% \[
% 0=\lim_{n\to\infty} i(c_n,x_1)=\lim_{n\to\infty}\frac{1}{\lambda}i(c_n,x_2)=i(x_1,x_2)>0
% \]
% which is a contradiction.
% \dots
% is a pA mapping class which fixes measured foliations $\mathcal F_1\neq \mathcal F_2$, then $\phi$ {\em does not} fix projectively the pair $(\mathcal F_1,\mathcal F_2)$. This follows at once by recalling that, by definition, there exists $\lambda\neq 1$ such that $\phi(\mathcal F_1)=\lambda\mathcal F_1$ and $\phi(\mathcal F_2)=\lambda^{-1}\mathcal F_2$.
% %So $\phi$ does not fix the projective class of $(\mathcal F,\mathcal F')$.

There is a natural continuous projection map $\pi\colon\mathfrak B\to \overline{\mathrm{Ind}(S)}$ defined as follows. For $x\in\text{Max}(S)$, consider the corresponding equivariant minimal Lagrangian $\widetilde\Sigma_x$. Then, $\pi(x)$ is the induced metric on $\widetilde\Sigma_x$. Otherwise, if $x\in\partial\mathfrak B$, consider the core of the tree corresponding to $x \in \Pp(\mathcal{ML}(S) \times \mathcal{ML}(S))$: its length spectrum coincides with that of a mixed structure $\mu$ on $S$. Set $\pi(x)=\mu$. This projection $\pi$ is continuous then by \cite[Theorem 6.13]{Charles_dPSL}. We consider the corresponding action of $\mathrm{MCG}(S)$ on $\overline{\mathrm{Ind}(S)}$ given by push-forward.

\begin{lemma} The action of $\mathrm{MCG}(S)$ on $\mathfrak B$ and $\overline{\mathrm{Ind}(S)}$ commute. In other words, for every $\phi\in\mathrm{MCG}(S)$
\[
\pi\circ \phi=\phi\circ \pi.
\]
\end{lemma}
\begin{proof} If $x=(x_1,x_2)$ is in the interior of $\mathfrak B$, then $\pi(\phi(x))=\phi(\pi(x))$ because $\phi(\pi(x))$ has the same length spectrum as the induced metric on the minimal Lagrangian associated to $\phi(x_1)$ and $\phi(x_2)$.  Suppose $x\in\partial\mathfrak B$ and consider a sequence $(x_n)_{n\in\mathbb N}\subset\mathrm{Max}(S)$ such that $x_n\to x$. Since, $\pi(\phi(x_n))=\phi(\pi(x_n))$ for all $n\in\mathbb N$, the result follows by continuity of $\phi$ and $\pi$.
\end{proof}

We are now ready to establish the main theorems of this section. In particular, Theorem \ref{thm:thmBagain} below is Theorem \ref{thm:actionMix} from the introduction.

\begin{thm}\label{thm:thmBagain} Assume $\phi\in\mathrm{MCG}(S)$ fixes $\mu\in\partial\overline{\mathrm{Ind}(S)}$. 
\begin{enumerate}
%	\item If $\mu$ is {\em purely laminar}, then $\phi$ is not periodic.
	\item If $\mu$ is {\em purely flat}, then $\phi$ is periodic.
	\item If $\mu$ is {\em properly mixed}, then $\phi$ is not pA.
\end{enumerate}
\end{thm}
\begin{proof} 
%Item (1) follows directly from Theorem \ref{thm:thurstonclass}.

For item (1), let us first assume by contradiction $\phi$ is pA. Since $\phi$ fixes $\mu$ projectively, there exist $\alpha>0$ such that $\phi(\mu)=\alpha \mu$ and up to replacing $\phi$ with $\phi^{-1}$, we can assume $\alpha\geq 1$. Now, there exists a measured foliation $\mathcal F$ and $\lambda>1$ such that $\phi(\mathcal F)=\lambda\mathcal F$. Recall that $\mu$ corresponds to a pair of filling measured foliations, thus $0\neq i(\mu,\mathcal F)$.  On the other hand,
\[
0\neq i(\mu,\mathcal F)=\frac{1}{\alpha}i(\phi(\mu),\mathcal F)=\frac{1}{\alpha}i(\mu,\phi^{-1}(\mathcal F))=\frac{1}{\alpha\lambda}i(\mu,\mathcal F)
\]
and we achieve a contradiction.

On the other hand, suppose by contradiction that $\mu$ is purely flat and $\phi$ is reducible. Then, there is a subsurface $T\subset S$ and $N\geq 0$ such that $\big(\phi^N\big)_{\big\vert_T}\colon T\to T$ and $\big(\phi^N\big)_{\big\vert_T}$ is pA. Applying the above argument to $\big(\phi^N\big)_{\big\vert_T}$ we achieve again a contradiction.

We establish item (2). Suppose $\mu$ is properly mixed, i.e. $\mu$ is not flat but it has at least one flat piece. We can decompose $S$ as
\[
\left(\{S_\alpha\}_{\alpha\in\mathcal A}, \{d_\beta\}_{\beta\in\mathcal B},\{\mu_\alpha\}_{\alpha\in\mathcal A}\right)
\]
where $\mu_\alpha$ is a flat structure or a (possibly-zero) laminar structure on $S_\alpha$ and $d_\beta$ is a maximal collection of closed geodesics so that
\[
i(d_\beta,d_{\beta'})=0\quad\text{and}\quad i(d_\beta,\mu)=0
\]
for all $\beta,\beta'\in\mathcal B$ and for every $c$ that intersects some $d_\beta$ transversely, $i(c,\mu)>0$. Note that there exists a unique set $\{d_\beta\}_{\beta\in\mathcal B}$ with these properties (see \cite[Theorem 1.1]{BIPP}).

\begin{claim}\label{important claim} $\phi$ fixes the set $\{d_\beta\}_{\beta\in\mathcal B}$.
\end{claim}
\begin{proof}[Proof of Claim] Observe that
\[
i(\phi(d_\beta),\phi(d_{\beta'}))=i(d_\beta,d_{\beta'})=0
\quad\text{and}\quad i(\mu,\phi(d_\beta))=i(\phi^{-1}(\mu),d_\beta)=0.
\]
If $c$ is a curve that intersects $\phi(d_\beta)$ transversely, then
\[
i(\phi^{-1}(c),d_\beta)=i(c,\phi(d_\beta))>0\quad\text{and}\quad
i(c,\mu)=i(\phi^{-1}(c),\phi^{-1}(\mu))>0.
\]
Thus, by uniqueness, $\{\phi(d_\beta)\}_{\beta\in\mathcal B}=\{d_\beta\}_{\beta\in\mathcal B}$.
\end{proof}
Thanks to the previous claim, we know that there exists $N>0$ such that for all $\alpha$ and $\beta$
\[
\phi^N(d_\beta)=d_\beta,\text{ and }\phi^N(S_\alpha)=S_\alpha.
\]
If $\phi^N$ is periodic, we are done. So let us assume that $\phi^N$ is not periodic. Then $\phi^N$ cannot be pA because $\mu$ is properly mixed, so there exists $\alpha$ such that $\mu_\alpha$ is flat and we can apply the argument in item (2) to $S_{\alpha}$. Therefore, $\phi^N$ is reducible.
\end{proof}

\begin{rmk} For an explicit example of $\mu$ purely flat and $\phi$ periodic such that $\phi(\mu)=\mu$, consider a singular flat metric on a surface of genus two obtained by doubling a singular flat metric on a torus with boundary.
\end{rmk}

\begin{thm}\label{thm: reducible properly mixed} Suppose $\phi\in\mathrm{MCG}(S)$ is reducible and fixes $\mu\in\partial\overline{\Ind(S)}$ which is properly mixed. Let $S=(S_\alpha,\{d_\beta\}_{\beta\in\mathcal B},\mu_\alpha)$ be the subdivision of $S$ induced by $\mu$.
\begin{enumerate}
    \item If for some $N>0$, $\psi_\alpha=\big(\phi^N\big)_{\big\vert_ {S_\alpha}}\colon S_\alpha\to S_\alpha$ is pA, then $\mu_\alpha=0$.
    \item If $\mu_\alpha\neq 0$ for all $\alpha\in\mathcal A$, then the canonical reduction system of $\phi$ is contained in $\{d_\beta\}_{\beta\in\mathcal B}$.
\end{enumerate}
\end{thm}
\begin{proof}
By hypotheses, we can decompose $S$ as $
\left(\{S_\alpha\}_{\alpha\in\mathcal A}, \{d_\beta\}_{\beta\in\mathcal B},\{\mu_\alpha\}_{\alpha\in\mathcal A}\right)$.
By Claim \ref{important claim}, there exists $N>0$ such that $\phi^N$ fixes $d_\beta$ for all $\beta\in\mathcal B$ and $\phi^N(S_\alpha)=S_\alpha$. Set $\psi_\alpha=(\phi^N)_{\big\vert S_\alpha}\colon S_\alpha\to S_\alpha$. 

In order to prove item (1), we need to consider three cases.
\begin{itemize}
	\item[a.] If $\mu_\alpha=0$, then $\psi_\alpha$ can be any element in $\text{MCG}(S_\alpha)$.
	\item[b.] If $(S_\alpha,\mu_\alpha)$ is purely flat (there exists at least one $\alpha$ for which this happens), then $\psi_\alpha$ can only be periodic by Theorem \ref{thm:thmBagain} item (1).
	\item[c.] Suppose $(S_\alpha,\mu_\alpha)$ is purely laminar. Since $\mu$ has a flat piece $\mu_\beta$, we know that $\psi_{\beta}$ is periodic and hence it fixes $\mu_\beta$ (not just projectively). We deduce that $\phi^{N}(\mu)=\mu$, otherwise we could find $z\neq 1$ such that $\psi_\alpha(\mu_\alpha)=z\mu_\alpha$, but then $\phi^N$ would not fix $\mu$ projectively. %Otherwise, we could find $\lambda>0$ such that for all curves $c_\alpha$ in $S_\alpha$ and $c_\beta$ in $S_\beta$. 
	%\begin{align*}
	%\lambda i(\mu, c_\alpha)+\lambda i(\mu, c_\beta)&=\lambda i(\mu, c_\alpha+c_\beta) \\
	%&=i(\phi^{N}(\mu),c_\alpha+c_\beta)\\
	%&=i(\phi^{N}(\mu_{\alpha}),c_\alpha)+(\phi^{N}(\mu_{\beta}),c_\beta) \\
	%&=zi(\mu_{\alpha},c_\alpha)+i(\mu_{\beta},c_\beta) 
	%\end{align*}
	We can now conclude that $\psi_\alpha$ cannot be pA. This is because if $c$ is a curve such that $i(\mu_\alpha,c)>0$, then 
	\[
	i(\mu_\alpha,c)=i(\psi_\alpha^{-1} (\mu_\alpha),c)=i(\mu_\alpha,\psi_\alpha (c))
	\]
	but $i(\mu_\alpha,\psi_\alpha (c))\neq i(\mu_\alpha,c)$ because $\psi_\alpha$ would change the length of curves transverse to $\mu_\alpha$. 
\end{itemize}
This completes the proof of item (1).

We wish to prove that the canonical reduction system $\{\bar c_1,\ldots,\bar c_k\}$ of $\phi$ is a subset  of $\{d_{\beta}\}_{\beta\in\mathcal B}$ under the additional assumption that $\mu_\alpha\neq 0$ for all $\alpha\in\mathcal A$. First, observe that by Claim \ref{important claim} $\{d_{\beta}\}_{\beta\in\mathcal B}$ is contained in a maximal reduction system for $\phi$. In particular $i(\bar c_j,d_\beta)=0$ for all $j$ and $\beta$. Moreover, since $\mu$ is properly mixed, there exists $\beta$ such that $\mu_\beta$ is flat, hence $\psi$ fixes $\mu$, not just its projective class, as observed before.

Assume, by contradiction $\bar c_j\not\in \{d_{\beta}\}_{\beta \in \mathcal{B}}$. Suppose $\bar c_j$ is contained in a purely flat piece $(S_\alpha,\mu_\alpha)$. Then, by Theorem \ref{thm:thmBagain} necessarily $\psi_\alpha$ is periodic. But this contradicts the property that if $i(\bar c_j,c)\neq 0$ and $n\neq 0$, then $\phi^n(c)\neq c$ since there exists $m$ such that $\psi_\alpha^m$ is the identity. Therefore $\bar c_j$ is contained in a purely laminar piece $\mu_\alpha$. 

By definition of $\{d_\beta\}_{\beta\in\mathcal B}$, $\psi_\alpha$ fixes a measured lamination $\mathcal F$ of full support in $S_\alpha$. 
If $\psi_\alpha$ is pA, then this contradicts item (1). Assume that $\psi_\alpha$ is periodic, so that there exists $m>0$ such that $\psi_{\alpha}^{m}$ is the identity. Then, we achieve again a contradiction because there would exist $c$ such that $i(\bar c_j,c)\neq 0$ but $\phi^m(c)=c$.

Finally, assume that $\psi_\alpha$ is reducible. We first observe that there needs to exist a subsurface $S_{\alpha,\beta}$ such that an opportune power of $\psi_\alpha$ restricted to $S_{\alpha,\beta}$ is pA. Otherwise, there exists $N\in\mathbb Z_{\geq 0}$ such that $\psi_\alpha^N=\text{id}$ and we obtain a contradiction by considering any closed curve $c\in S_\alpha$ which intersects $\bar c_j$ transversely. Now, let $S_{\alpha,\beta}$ be a subsurface of $S_\alpha$ and $\psi_{\alpha,\beta}$ a power of $\psi_\alpha$ such that $\psi_{\alpha,\beta}\big\vert_{S_{\alpha,\beta}}$ is pA. Let $\mathcal F^s$ be the unstable measured foliation of $\psi_{\alpha,\beta}$ with stretch factor $z>1$. But then, as we observed earlier, $\psi_\alpha(\mathcal F)$ fixes the measured lamination (not just its projective class) and, since $\mathcal F$ has full support, there exists an integer $M>0$ such that
\[
0\neq i(\mathcal F,\mathcal F^s)=i(\psi_\alpha^M(\mathcal F),\mathcal F^s)=i(\mathcal F,\psi_\alpha^{-M}(\mathcal F^s))=\frac{1}{z}i(\mathcal F,\mathcal F^s),
\]
which gives a contradiction.
\end{proof}

\section{$\overline{\mathfrak{a}^{+}}$-valued measured laminations and mixed structures}\label{sec:weylcores}

In this final section we introduce Weyl chamber valued measured laminations and use them to refine the notion of mixed structures on a closed surface defined in \cite{DLR_flat}, and generalized to higher order differentials in \cite{OT_Sp4,OT}. We show that the core of the product of two trees dual to measured laminations is dual to such a mixed structure, thus giving a new interpretation of the boundary objects in our compactification of $\Max(S)$. 

Let $\mathfrak{g}$ be a real semisimple Lie algebra. The choice of a maximal compact subalgebra $\mathfrak{k}$ induces an orthogonal decomposition of $\mathfrak{g}$ for the Killing form: 
\[
    \mathfrak{g}=\mathfrak{k}\oplus\mathfrak{m} \ .
\]
A Cartan subalgebra $\mathfrak{a} \subset \mathfrak{g}$ is a maximal abelian subspace of $\mathfrak{m}$. This induces a decomposition of $\mathfrak{g}$ in $\ad(\mathfrak{a})$-eigenspaces
\[
    \mathfrak{g} = \mathfrak{g}_{0} \bigoplus _{\alpha \in \Sigma} \mathfrak{g}_{\alpha} \ .
\]
Elements of $\Sigma \subset \mathfrak{a}^{*}=\Hom(\mathfrak{a}, \R)$ are called restricted roots of $\mathfrak{a}$ in $\mathfrak{g}$. Here we can extract a subset $\Delta$ of \emph{simple} roots with the property that any $\alpha \in \Sigma$ can be expressed as linear combinations of simple roots with coefficients all of the same sign. This distinguishes, thus, a subset of positive roots that we denote by $\Sigma^{+}\subset \Sigma$. The closed positive Weyl chamber of $\mathfrak{a}$ associated to $\Sigma^{+}$ is then the cone
\[
    \overline{\mathfrak{a}^{+}} = \{ X \in \mathfrak{a} \ | \ \alpha(X)\geq 0 \ \ \forall \alpha \in \Sigma^{+}\} \ .
\]
We also denote by $W$ the Weyl group of $\mathfrak{g}$, i.e. $W=N(\mathfrak{a})/\mathfrak{a}$. 

\begin{deff}\label{def:Weyl_laminations} An $\overline{\mathfrak{a}^{+}}$-valued measured lamination on a (not necessarily closed) surface $S$ is a geodesic lamination $\lambda$ on $S$ that supports a measure $\mu$ on transverse arcs that takes value in $\overline{\mathfrak{a}^{+}}$ and satisfies the following properties:
\begin{enumerate}[label=\alph*)]
    \item $\mu(\gamma) \neq 0$, if $\gamma$ intersects $\lambda$ transversely;
    \item if $\gamma$ and $\gamma'$ are homotopic arcs transverse to $\lambda$ and there is a homotopy between them that preserves transversality at every time, then $\mu(\gamma)=\mu(\gamma')$;
    \item $\mu$ is additive on concatenation of paths, i.e. $\mu(\gamma\gamma')=\mu(\gamma)+\mu(\gamma')$ for all $\gamma$ and $\gamma'$ transverse to $\lambda$ such that concatenation is defined.
    % \GM{In the case $\mathrm{SL}(3,\mathbb{R})$, this should correspond to subsurfaces where, if the complete $\lambda$ to a maximal lamination of the subsurface, the triple ratios are not growing as fast as the shearing parameters.}
\end{enumerate}
\end{deff}

\begin{rmk}\label{rmk:examples} If $\mathfrak{g}=\mathfrak{sl}(2,\R)$, then we can identify the closed positive Weyl chamber with $\R_{\geq 0}$. Thus, in this case, Definition \ref{def:Weyl_laminations} recovers the standard notion of measured laminations. Similarly, if $\mathfrak{g}=\mathfrak{sl}(2,\R)\oplus \mathfrak{sl}(2,\R)$, then $\overline{\mathfrak{a}^{+}}$-valued laminations can be identified with ordered pairs $(\lambda_{1}, \lambda_{2})$, such that $\lambda_{1}$, $\lambda_{2}$ and $\lambda_{1}+\lambda_{2}$ are measured laminations (i.e. $\lambda_{1}$ and $\lambda_{2}$ are nowhere transverse).
\end{rmk}

We can also extend the classical notion of trees dual to a measured lamination to this context.

\begin{deff}\label{def:dual_vector_trees} Let $(T,d)$ be an $\mathbb{R}$-tree acted upon by the fundamental group of $S$. We say that the action of $\pi_{1}(S)$ is dual to an $\overline{\mathfrak{a}^{+}}$-valued measured lamination $\mu$, if there is an equivariant map $p:\widetilde{S} \rightarrow T$ and an $\overline{\mathfrak{a}^{+}}$-valued distance $d_{\mathfrak{a}^{+}}: T\times T \rightarrow \overline{\mathfrak{a}^{+}}$ such that
\begin{enumerate}[label=\alph*)]
    \item for all $x, y \in \widetilde{S}$, we have $d_{\mathfrak{a}^{+}}(p(x),p(y))=\mu(\gamma)$ for some (hence any) arc $\gamma:[0,1] \rightarrow \widetilde{S}$ transverse to the support of $\mu$ with $\gamma(0)=x$ and $\gamma(1)=y$;
    %\item $d_{\mathfrak{a}^{+}}$ is additive on geodesics in $(T,d)$;
    \item given a geodesic path $\gamma:[0,1] \rightarrow T$, we have $d(\gamma(0), \gamma(1))\geq \| d_{\mathfrak{a}^{+}}(\gamma(0), \gamma(1))\|$. Here $\| \cdot \|$ denotes the standard Euclidean norm of a vector in $\overline{\mathfrak{a}^{+}}$
\end{enumerate}
\end{deff}

We now combine $\overline{\mathfrak{a}^{+}}$-valued measured laminations with the classical notion of $1/k$-translation surfaces in order to define a hybrid structure on $S$. 

\begin{deff}\label{def:mixed_structures} Let $\overline{\mathfrak{a}^{+}}$ be a closed Weyl chamber and $k\geq 1$ an integer. An $(\overline{\mathfrak{a}^{+}}, k)$-mixed structure on a closed surface $S$ is the datum of
\begin{enumerate}
    \item a collection of non-homotopically trivial, pairwise non-homotopic, disjoint simple closed curves $\gamma_{1}, \dots , \gamma_{n}$ on $S$;
    \item for each connected component $S'$ of $S\setminus \cup_{j} \gamma_{j}$ either
        \begin{itemize}
            \item an $\overline{\mathfrak{a}^{+}}$-valued measured lamination $\lambda$, where we allow each $\gamma_{j}$ to be in the support; or
            \item a meromorphic $k$-differential of finite area that endows $S'$ with a $1/k$-translation surface structure.
        \end{itemize}
\end{enumerate}
\end{deff}

These $(\overline{\mathfrak{a}^{+}}, k)$-mixed structures can be interpreted as dual to the $(\mathfrak{a}, W)$-complexes studied by Anne Parreau (\cite{parreau_invariant}) in the context of $\mathfrak{g}=\mathfrak{sl}(3,\R)$. Let us recall briefly how these complexes are defined and explain in which sense these notions can be considered dual to each other.

Following \cite{parreau_invariant}, an $(\mathfrak{a}, W)$-complex $K$ is the union of (possibly degenerate) polygons in $\mathfrak{a}$ glued together along boundary segments via elements of $W_{aff}=W\rtimes \R$. More precisely, there is a family of affine simplices $P_{\mu} \subset \mathfrak{a}$ and injective maps $\phi_{\mu}: P_{\mu} \rightarrow K$ such that if $K_{\mu}=\phi_{\mu}(P_{\mu})$ and $K_{\mu'}=\phi_{\mu'}(P_{\mu'})$ have non-empty intersection then there is $w_{\mu,\mu'} \in W_{aff}$ such that $\phi_{\mu}(x)=\phi_{\mu'}(x')$ if and only if $x'=w_{\mu, \mu'}(x)$ and $P_{\mu} \cap w_{\mu, \mu'}^{-1}(P_{\mu'})$ is a face in $P_{\mu}$. We only consider connected and simply-connected $(\mathfrak{a}, W)$-complexes acted upon by $\pi_{1}(S)$. Note that, since the gluing maps between simplices are Euclidean isometries, the Euclidean distance on $\mathfrak{a}$ induces a distance on $K$. We will only work with $(\mathfrak{a}, W)$-complexes whose induced distance is CAT(0). Similarly, $K$ is also endowed with an $\overline{\mathfrak{a}^{+}}$-valued distance inherited from $\mathfrak{a}$. 

Examples of $(\mathfrak{a}, W)$-complexes are subcomplexes of an Euclidean building modelled on $W_{aff}$. We will see that cores of product of two trees dual to measured laminations are indeed $(\mathfrak{a}, W)$-complexes where $\mathfrak{a}$ is the Cartan subalgebra of $\mathfrak{sl}(2,\R)\oplus \mathfrak{sl}(2,\R)$ and $W=\{\pm \mathrm{Id}\}$. 

\begin{deff}\label{def:dual-mixed} We say that an $(\mathfrak{a}, W)$-complex $K$ acted upon by $\pi_{1}(S)$ is dual to an $(\overline{\mathfrak{a}^{+}}, k)$-mixed structure $\mu$ on $S$ if we can decompose $K$ into a $1$-dimensional part $K_{1}$ and a $2$-dimensional part $K_{2}$ such that 
\begin{itemize}
    \item $K_{1}$ is the union of $\R$-trees dual to the laminar part of $\mu$;
    \item $K_{2}$ is endowed with a $1/k$-translation surface structure isomorphic to the universal cover of the flat parts of $\mu$.
\end{itemize}
\end{deff}

Note that the $2$-dimensional part of an $(\mathfrak{a}, W)$-complex can be endowed with a $1/k$ translation surface structure only if $W$ contains the subgroup generated by rotations of angle $\frac{2\pi}{k}$.

We believe that these mixed structures naturally appear in a harmonic map compactification of the Hitchin and maximal components of the character variety for real Lie groups $G$ of rank $2$. In this context, Labourie (\cite{Labourie_cyclic}), Collier (\cite{collier2016maximal}) and Collier-Tholozan-Toulisse (\cite{collier2019geometry}) proved that given a Hitchin or maximal representation $\rho:\pi_{1}(S) \rightarrow G$ there exists a unique $\rho$-equivariant minimal surface $\widetilde\Sigma_{\rho}$ in $G/K$, where $K$ is a maximal compact subgroup of $G$. One could then find a compactification of these components by studying the limiting behaviour of $\widetilde\Sigma_{\rho_n}$ when $\rho_{n}$ leaves all compact sets in the character variety. Up to subsequences, and after rescaling the metric on $G/K$ appropriately, $\widetilde\Sigma_{\rho_n}$ should converge to a subcomplex $\widetilde\Sigma_{\infty}\subset B$, where $B$ is a non-discrete Euclidean building modelled on the affine Weyl-group of $G$. We conjecture that $\Sigma_{\infty}$ is dual to a mixed structure as in Definition \ref{def:dual-mixed} where $\overline{\mathfrak{a}^{+}}$ is a Cartan subalgebra of the Lie algebra of $G$ and $k$ depends on the particular group. More precisely, we conjecture the following:

\newpage

\begin{conj}\label{conj} \hspace{-1cm} \begin{enumerate}[label=\alph*)] 
\item Let $G$ be a real split semisimple Lie group of rank $2$. Then the boundary of $\Hit(S,G)$ can be identified with the space of projective classes of $(\overline{\mathfrak{a}^{+}}, k)$-mixed structures where:
    \begin{itemize}
        \item if $G=\SL(3,\R)$, then $\mathfrak{a}=A_{2}$ and $k=3$;
        \item if $G=\Sp(4,\R)$, then $\mathfrak{a}=B_{2}$ and $k=4$;
        \item if $G=G_{2}^{\R}$, then $\mathfrak{a}=G_{2}$ and $k=6$.
    \end{itemize}
\item Let $G$ be a real semisimple Lie group of Hermitian type and rank $2$. Then the boundary of $\Max(S,G)$ can be identified with the space of projective classes of $(\overline{\mathfrak{a}^{+}}, k)$-mixed structures where:
    \begin{itemize}
        \item if $G=\SL(2,\R)\times \SL(2,\R)$, then $\mathfrak{a}=A_{1}\times A_{1}$ and $k=2$;
        \item if $G=\SO(2,n)$ with $n\geq 3$, then $\mathfrak{a}=B_{2}$ and $k=4$.
    \end{itemize}
\end{enumerate}
\end{conj}

In support of this conjecture, we show that the core of the product of two trees dual to measured laminations is dual to an $(\overline{A_{1}^{+}\times A_{1}^{+}}, 2)$-mixed structure and that we can identify $\Core(\mathcal{T}, \mathcal{T})$ with the space of such structures, thus proving the conjecture for $G=\SL(2,\R)\times \SL(2,\R)$. Moreover, in a subsequent paper (\cite{LTW_SL3}), the third author, in collaboration with Loftin and Wolf, will give further evidence towards Conjecture \ref{conj} by describing the geometry of the harmonic maps to buildings arising from some diverging sequences of $\SL(3,\R)$-Hitchin representations. It would be interesting to introduce a higher rank version of our vector valued mixed structures, at least for the case of $\SL(d,\R)$-Hitchin components, and relate it to the subspaces of the Euclidean building studied in \cite{Le2016} and \cite{martone2018}.

\begin{lemma}\label{lm:core_as_mixed} Let $T_{1}$ and $T_{2}$ be real trees dual to measured laminations $\lambda_{1}$ and $\lambda_{2}$ and let $C$ be the core of $T_{1}\times T_{2}$. Then $C$ is an $(A_1\times A_1,\{\pm \text{Id}\})$-complex dual to an $(\overline{A_{1}^{+}\times A_{1}^{+}}, 2)$-mixed structure on $S$.
\end{lemma}
\begin{proof} We already saw in Section \ref{sec:cores} that $C$ is the union of a $1$-dimensional subcomplex $C_{1}$ and a $2$-dimensional subcomplex $C_{2}$ of $T_{1}\times T_{2}$. Moreover, we showed that each connected component of $C_{2}$ is the universal cover of a half-translation surface structure on a subsurface $S'$ of $S$, on which the laminations $\lambda_{1}$ and $\lambda_{2}$ fill. Thus, it only remains to show that each connected component $C_{1}'$ of $C_{1}$ is a tree dual to an $\overline{A_{1}^{+}\times A_{1}^{+}}$-valued measured lamination. \\
Recall from Section \ref{sec:cores} that $C_{1}'$ is the image under the map $F$ defined in Equation \eqref{eq:F} of a domain $\Omega \subset \h^{2}$ that can be identified with the universal cover of a subsurface $S'$ of $S$ on which $\lambda_{1}$ and $\lambda_{2}$ are nowhere transverse. Moreover, we observe that $C_{1}'$ has a natural distance $d$ induced by the ambient space
\[
    d((x_{0}, y_{0}), (x_{1},y_{1}))=\sqrt{d_{1}(x_{0}, y_{0})^{2}+ d_{2}(x_{1},y_{1})^{2}}
\]
and a natural $\overline{A_{1}^{+}\times A_{1}^{+}}$-valued distance $\vec{d}$ defined by 
\[
    \vec{d}((x_{0}, y_{0}), (x_{1},y_{1}))=(d_{1}(x_{0}, y_{0}), d_{2}(x_{1},y_{1})) \ .
\]
We claim that $(C_{1}', d)$ is an $\R$-tree dual to the $\overline{A_{1}^{+}\times A_{1}^{+}}$-valued measured lamination $\vec{\lambda}=(\lambda_{1}, \lambda_{2})$ (see Remark \ref{rmk:examples}).
By Lemma \ref{lm:1-dim core}, $C_{1}'$ can be identified with the $\R$-tree dual to the measured lamination $\lambda_{0}=\lambda_{1}+\lambda_{2}$, if endowed with the distance $d_{0}$ introduced in Section \ref{sec:cores}. In particular, there is a continuous $\pi_{1}(S')$-equivariant map $p:=p_{\lambda_{0}}: \Omega \rightarrow C_{1}'$. It follows immediately from the definitions and the fact that $T_{1}$ and $T_{2}$ are dual to the laminations $\lambda_{1}$ and $ \lambda_{2}$ that for all $x,y \in \Omega'$ we have
\[
    \vec{d}(p(x), p(y))=\vec{\lambda}(\gamma)
\]
for all $\gamma:[0,1]\rightarrow \Omega$ transverse to the support of $\lambda_{0}$ with $\gamma(0)=x$ and $\gamma(1)=y$. \\
Property b) in Definition \ref{def:dual_vector_trees} also holds. Indeed, a geodesic path $\gamma=(\gamma_{1},\gamma_{2}):[0,1] \rightarrow C_{1}'\subset T_{1}\times T_{2}$, seen in the quadrant $\gamma_{1}\times \gamma_{2}$, consists of a concatenation of horizontal, vertical or diagonal paths in which the projections onto the two factors are always non-decreasing. Hence, 
\[
    d(\gamma(0), \gamma(1))\geq \|\vec{d}(\gamma(0), \gamma(1))\| \ ,
\]
and the proof is complete.
\end{proof}

\begin{thm} The space of $(\overline{A_{1}^{+}\times A_{1}^{+}}, 2)$-mixed structures on $S$ is homeomorphic to $\Core(\mathcal{T}, \mathcal{T})$. 
\end{thm}
\begin{proof} Let $Y$ denote the set of $(\overline{A_{1}^{+}\times A_{1}^{+}}, 2)$-mixed structures on $S$. We still need to define a topology on $Y$. We will construct a bijection
\[
    \varphi: Y \rightarrow \mathcal{ML}(S)\times \mathcal{ML}(S)
\]
with the property that for all $y \in Y$ the core of the product of trees corresponding to $\varphi(y)$ is dual to the $(\overline{A_{1}^{+}\times A_{1}^{+}}, 2)$-mixed structure $y$. We then give $Y$ the topology that makes $\varphi$ a homeomorphism, thus proving the result. \\
Given $y \in Y$, let $\gamma_{1}, \dots, \gamma_{n}$ be the simple closed curves subdividing $S$ into its laminar and flat parts, as in Definition \ref{def:mixed_structures}. Let $S_{i}$ for $i=1, \dots, m$ denote the connected components of $S\setminus \cup_{j} \gamma_{j}$. If $S_{i}$ is endowed with a half-translation surface structure induced by a meromorphic quadratic differential $q_{i}$ of finite area, then the horizontal and vertical foliations of $q_{i}$ determine a pair of measured laminations $(\lambda^{i}_{1}, \lambda^{i}_{2})$. Here we are implicitly using the well-known homeomorphism between the space of measured foliations arising this way and the space of measured laminations, see for instance (\cite{Levitt_foliations=laminations}, \cite{LM_ergodic_measured_laminations}). On the other hand, by Remark \ref{rmk:examples}, if $S_{i}$ carries an $\overline{\mathfrak{a}^{+}}$-valued measured lamination, then this is equivalent to a pair of measured laminations $(\lambda_{1}^{i}, \lambda_{2}^{i})$ possibly containing some boundary curves $\gamma_{j}$ in their support. We can then associate to $y \in Y$ the pair of measured laminations $(\lambda_{1}, \lambda_{2}) \in \mathcal{ML}(S) \times \mathcal{ML}(S)$ defined as $\lambda_j=\sum_{i}^{m} \lambda_j^{i}$ for $j=1,2$. Since the horizontal and vertical measured foliations uniquely determine a meromorphic quadratic differential of finite area (\cite{GM_extremal_length}), using Remark \ref{rmk:examples} and Lemma \ref{lm:decomposition_pair_laminations}, it is clear that $\varphi$ is a bijection. \\
Moreover, comparing the definition of the map $\varphi$ with Lemma \ref{lm:core_as_mixed}, it is easy to verify that the core of the product of trees dual to the pair $\varphi(y)$ is dual to the $(\overline{A_{1}^{+}\times A_{1}^{+}}, 2)$-mixed structure $y$ we started with. 
\end{proof}

\bibliographystyle{alpha}
\bibliography{bs-bibliography}

\bigskip
\noindent \footnotesize \textsc{Department of Mathematics, University of Michigan, Ann Arbor}\\
\emph{E-mail address:}  \verb|martone@umich.edu|

\bigskip
\noindent \footnotesize \textsc{Department of Mathematics and Statistics, University of Massachusetts, Amherst}\\
\emph{E-mail address:}  \verb|ouyang@math.umass.edu|

\bigskip
\noindent \footnotesize \textsc{Department of Mathematics, Rice University}\\
\emph{E-mail address:} \verb|andrea_tamburelli@libero.it|

\end{document}